\documentclass[10pt]{amsart}
\usepackage{amssymb}
\usepackage[T1]{fontenc}

\newtheorem*{theorem}{Theorem}
\newtheorem*{corollary}{Corollary}

\begin{document}

\title{A Converse to the Whitehead Theorem}
\author{Pasha Zusmanovich}
\address{Hl\'i{\dh}arhjalli 62, K\'opavogur 200, Iceland}
\email{justpasha@gmail.com}
\date{August 1, 2008; last revised May 19, 2009}

\begin{abstract}
In this note, which is a postscript to the earlier note \cite{h2}, 
we show that finite-dimensional Lie algebras over a field of characteristic zero
such that their high-degree cohomology in any finite-dimensional non-trivial 
irreducible module vanishes, are, essentially, 
direct sums of semisimple and nilpotent algebras.
\end{abstract}

\maketitle

The classical First and Second Whitehead Lemmata state that the first, respectively 
second, cohomology group of a finite-dimensional semisimple Lie algebra 
over a field of characteristic zero with coefficients
in any finite-dimensional module vanishes. This pattern breaks down, however, at the 
third cohomology: it is well-known that the third cohomology of a simple Lie algebra
with coefficients in the trivial one-dimensional module is one-dimensional, 
the basic cocycle being constructed from the Killing form.
There is, however, a related and not less classical result 
holding for all higher cohomology groups which is called sometimes the Whitehead Theorem:
for any finite-dimensional semisimple Lie algebra $L$ and any non-trivial
finite-dimensional irreducible $L$-module $V$, $H^n(L,V) = 0$ for any $n\ge 3$.

It is very natural to ask whether a converse to these statements holds.
A converse to the Second Whitehead Lemma was established in \cite{h2}.
The aim of this note is to observe that a converse to the Whitehead Theorem 
readily follows from the results already established in the literature.

In what follows, all algebras and modules assumed to be finite-dimensional,
and the base field is of characteristic zero. 
Finite-dimensionality is obviously crucial
in most of the places. Some reasonings below are valid in any characteristic, but the 
case of positive characteristic is trivial modulo existing results, 
as noted in \cite{h2}.

Our notations and terminology are standard: for a Lie algebra $L$ and an $L$-module $V$,
$H^n(L,V)$ denotes the $n$th cohomology of $L$ with coefficients in $V$, 
$V^L$ denotes the submodule of $L$-invariant points,
$Z(L)$ and $Rad(L)$ denotes the center and the solvable radical of $L$ respectively, 
$V^*$ denotes an $L$-module adjoint to $V$.
When considered as an $L$-module, the base field $K$ is understood as the trivial 
one-dimensional module. 
$L$ is called unimodular if $Tr(ad\,x) = 0$ for any $x\in L$, where
$Tr$ denotes the trace of a linear operator.

\begin{theorem}
For a Lie algebra $L$, the following are equivalent:
\begin{enumerate}
\item
$L$ is the direct sum of a semisimple algebra and a nilpotent algebra.
\item
$H^n(L,V) = 0$ for any $n$ and any non-trivial irreducible $L$-module $V$.
\item
$H^{\dim L - 1}(L,V) = 0$ for any non-trivial irreducible $L$-module $V$.
\item
$H^1(L,V) = 0$ for any non-trivial irreducible $L$-module $V$.
\end{enumerate}  
\end{theorem}

\begin{proof}
(i) $\Rightarrow$ (ii):
For a nilpotent Lie algebra even a more general statement is true:  
by \cite[Th\'eor\`eme 1]{dixmier}, cohomology of a nilpotent Lie algebra with 
coefficients in a module without trivial submodules, vanishes.
The same is true for a semisimple Lie algebra:
this follows from the classical results mentioned above,
complete reducibility of representations of a semisimple Lie algebra,
and the additivity of cohomology (cohomology with coefficients in the direct sum of 
modules decomposes as the direct sum of cohomologies with coefficients in direct 
summands).

Now let $L = S \oplus N$, the direct sum of a semisimple Lie algebra $S$ and a nilpotent
Lie algebra $N$, and $V$ be a non-trivial irreducible $L$-module. Suppose that $V$, 
as an $S$-module, contains the trivial submodule $K$. As the actions of $S$ and $N$ on $V$
commute, $S(N^i K) = N^i(SK) = 0$ for any $i\in \mathbb N$. 
Thus $\sum_i N^i K$ is an $L$-submodule
of $V$ and hence coincides with $V$, so $V$ is a trivial $S$-module.
Then any trivial $N$-submodule of $V$ will be also a trivial $L$-submodule of $V$,
hence $V$ does not contain trivial submodules as an $N$-module.

Interchanging in this reasoning $S$ and $N$ (we have not used any properties of 
$S$ and $N$, and it is valid for the direct sum of any Lie algebras), 
we see that at least for one of algebras 
$S$, $N$, the module $V$ over that algebra does not contain trivial submodules.

By the K\"unneth formula,
$H^n(S \oplus N, V) \simeq \bigoplus_{i+j=n} H^i(S,V) \otimes H^j(N,V)$,
and by above, at least one of the tensor factors in each summand vanishes.

\medskip
(ii) $\Rightarrow$ (iii), as well as (ii) $\Rightarrow$ (iv), is obvious.

\medskip
(iii) $\Rightarrow$ (iv):
By the result of Hazewinkel \cite{haze}, which is a generalization
of the well-known Poincar\'e duality for the cohomology of unimodular
Lie algebras (\cite[Chapter 1, \S 1.6B]{fuchs}) to all Lie algebras,
$H^n(L, (V^{tw})^*) \simeq H^{\dim L - n} (L,V)^*$ for any 
$0 \le n \le \dim L$ and any $L$-module $V$, where $V^{tw}$ is a ``twisted''
$L$-module defined on the vector space $V$ as follows:
$x \mapsto \rho_V(x) - Tr(ad\,x)1_V$ for $x\in L$, where $x \mapsto \rho_V(x)$ is the 
representation corresponding to the module $V$, and $1_V$ is the identity map on 
$V$.

Let $V$ be a non-trivial irreducible $L$-module. By the Hazewinkel's duality,
\begin{equation}\label{haze}
H^1(L, (V^{tw})^*) \simeq H^{\dim L - 1} (L,V)^*
\end{equation}
and
\begin{equation}\label{haze-minustw}
H^1(L,V) \simeq H^{\dim L - 1} (L,(V^*)^{-tw})^*,
\end{equation}
where the ``minus-twisted'' module 
$W^{-tw}$ is obtained from the module $W$ by the formula 
$x \mapsto \rho_W(x) + Tr(ad\,x)1_W$.
It is well-known (see, for example, \cite[Chapter 1, \S 3.3]{bourbaki}) that 
a module $V$ is irreducible if and only if $V^*$ is irreducible,
and it is obvious that $V$ is irreducible if and only if $V^{-tw}$ is irreducible.
Hence $(V^*)^{-tw}$ is an irreducible $L$-module. 

If $(V^*)^{-tw}$ is a trivial module,
then $V^*$ is an one-dimensional module $Kv$ defined by $xv = -Tr(adx)v$ for $x\in L$,
so $V$ is an one-dimensional module $Kv$ defined by $xv = Tr(adx)v$, and
$V^{tw}$, as well as $(V^{tw})^*$, is the one-dimensional trivial $L$-module. 
Since $H^{\dim L - 1}(L,V) = 0$,
(\ref{haze}) implies $H^1(L,K) = 0$, what is equivalent to $[L,L] = L$.
The latter, in its turn, implies that $Tr(ad\,x) = 0$ for any $x\in L$
(i.e. $L$ is unimodular), so $V$ is trivial, a contradiction.

Thus $(V^*)^{-tw}$ is a non-trivial module, hence $H^{\dim L - 1} (L,(V^*)^{-tw}) = 0$,
and by (\ref{haze-minustw}), $H^1(L,V) = 0$.

Of course, very similar reasonings could be utilized to prove the 
reverse implication (iv) $\Rightarrow$ (iii).

\medskip
(iv) $\Rightarrow$ (i):
In \cite[Theorem 3]{barnes} a very close statement is proved: a Lie algebra $L$ is the
direct sum of a semisimple algebra and a supersolvable algebra if and only if 
$H^1(L,V) = 0$ for any irreducible $L$-module $V$ of dimension $>1$. 
Note a subtle but important 
difference between two conditions on $L$-modules: solvable Lie algebras
(and, more generally, Lie algebras with a nonzero radical) may have 
(and necessarily have if the ground field is algebraically closed, by the Lie Theorem)
one-dimensional (and, hence, irreducible) non-trivial module. 
What follows is, essentially, modified for our needs proof from \cite{barnes}.

First note that condition (iv) is closed under quotients. Indeed, let a Lie algebra
$L$ satisfies (iv), $I$ be an ideal of $L$, and $V$ be a non-trivial irreducible 
$L/I$-module. We can lift $V$ to an $L$-module by letting $I$ act on it trivially. 
Obviously, the such lifted $L$-module is also non-trivial and irreducible.
The beginning of the standard $5$-term long exact sequence which follows from the 
Hochschild--Serre spectral sequence, reads: $0 \to H^1(L/I, V^I) \to H^1(L,V)$.
As $V^I = V$ and $H^1(L,V) = 0$, we have $H^1(L/I, V) = 0$.

Now we argue by induction on the dimension of $L$. The case of dimension $1$ is obvious.
Suppose $L$ is not semisimple and satisfies (iv). 
Then $L$ has a minimal abelian ideal $I$, 
$L/I$ also satisfies (iv), and by induction assumption satisfies (i).
Let $L = S \ltimes Rad(L)$ be the Levi--Malcev decomposition of $L$. Obviously,
$I \subseteq Rad(L)$, and $L/I \simeq S \ltimes Rad(L)/I$ is the Levi--Malcev
decomposition of $L/I$. As (i) amounts to saying that the Levi--Malcev decomposition of 
a Lie algebra degenerates to the direct sum and its radical is nilpotent, it follows that
$L/I \simeq S \oplus Rad(L)/I$ and $Rad(L)/I$ is nilpotent.
Note also that since $I$ is minimal, it is an irreducible $L/I$-module.

If $I$ is a trivial $L/I$-module, then $I \subseteq Z(L) \subseteq Z(Rad(L))$.
Since extension of a nilpotent Lie algebra by an ideal lying in its center is obviously
nilpotent, $Rad(L)$ is nilpotent. Also, since $S$ acts on $Rad(L)/I$ trivially,
$[S,Rad(L)] \subseteq I$, and hence $S$ acts on $Rad(L)$ nilpotently. But since $S$
is semisimple, this means $[S,Rad(L)] = 0$.

Suppose $I$ is not a trivial $L/I$-module.
By (i) $\Rightarrow$ (ii), $H^2(L/I,I) = 0$, and hence
the extension $0 \to I \to L \to L/I \to 0$ splits, i.e. 
$L \simeq S \oplus Rad(L)/I \oplus I$, thus $L$ is isomorphic to the direct sum 
of a semisimple algebra $S$ and a nilpotent algebra $Rad(L)/I \oplus I$.
\end{proof}

\begin{corollary}
For a Lie algebra $L$, the following are equivalent:
\begin{enumerate}
\item
$H^n(L,V) = 0$ for any $n\ge 3$ and any non-trivial irreducible $L$-module $V$.
\item
$L$ is either the direct sum of a semisimple and a nilpotent algebra,
or a $2$-dimensional algebra, or a $3$-dimensional unimodular algebra.
\end{enumerate}
\end{corollary}

\begin{proof}
For $\dim L > 3$, in one direction, put $n = \dim L - 1$ and invoke implication 
(iii) $\Rightarrow$ (i) of the Theorem, and in another direction
this is exactly implication (i) $\Rightarrow$ (ii) of the Theorem.

For $\dim L = 2$, (i) is satisfied vacuously.

Let $\dim L = 3$. (i) is non-vacuous only for $n=3$.
By the Hazewinkel's duality, 
$H^3(L,V) \simeq H^0(L,(V^*)^{-tw})^* \simeq (((V^*)^{-tw})^L)^*$. By the same arguments 
as in the proof of implication (iii) $\Rightarrow$ (iv) of the Theorem,
$(V^*)^{-tw}$ is irreducible if and only if $V$ is irreducible, and is trivial 
if and only if $V$ is isomorphic to $K^{-tw}$. 
Consequently, (i) in this case is equivalent to the triviality of the module $K^{-tw}$, 
what, in its turn, is equivalent to unimodularity of $L$.
\end{proof}

The implication (i) $\Rightarrow$ (ii) of the Corollary is the converse to the 
Whitehead Theorem.

Using the well-known classification of $3$-dimensional Lie algebras (see, for example,
\cite[Chapter 1, \S 6, Exercise 23]{bourbaki}), it is easy to list $3$-dimensional 
unimodular Lie algebras. The list consists of the abelian algebra, 
the nilpotent (Heisenberg) algebra, simple algebras, and solvable algebras
with basis $\{x,y,z\}$ and multiplication table $[x,y] = 0$, $[x,z] = ax + by$
$[y,z] = cx - ay$ for certain $a,b,c\in K$ with $a^2 + bc \ne 0$.

\section*{Acknowledgements}

I am grateful to Dietrich Burde for interesting discussion, and 
to the anonymous referee for pointing out few rough edges in the previous
version of the manuscript.

\end{document}